\documentclass[11pt]{amsart}

\usepackage[all]{xy}
\usepackage{amsthm} 
\usepackage{amssymb}
\usepackage{graphicx}									 %Include graphics
\usepackage{paralist}
\usepackage[british]{babel}
\usepackage{color}
\usepackage{hyperref}
\usepackage{pdfpages}

\newtheorem{thm}{Theorem}[section]
\newtheorem*{thmw}{Theorem}

\newtheorem{prp}[thm]{Proposition}
\newtheorem{lem}[thm]{Lemma}

\newtheorem{dfn}{Definition}[section]

\newcommand{\R}{\mathbb{R}}

\newcommand{\Z}{\mathbb{Z}}

\newcommand{\To}{\rightarrow}

\newcommand{\MTo}{\mapsto}

\title[Equivariant smoothing of piecewise linear manifolds]{Equivariant smoothing of piecewise linear manifolds}
\author{Christian Lange}

\address{Christian Lange, Mathematisches Institut der Universit\"at zu K\"oln, Weyertal 86-90, 50931 K\"oln, Germany}
\email{clange@math.uni-koeln.de}
\thanks{Supported by a `Kurzzeitstipendium f\"{u}r Doktoranden' by the German Academic Exchange Service (DAAD). The content of this paper is part of the author's thesis written at the University of Cologne \cite{Lange}.}
\subjclass{57M50, 57Q91, 57R10}

\begin{document}

\begin{abstract} We prove that every piecewise linear manifold of dimension up to four on which a finite group acts by piecewise linear homeomorphisms admits a compatible smooth structure with respect to which the group acts smoothly.
\end{abstract}
\maketitle	

\section{Introduction}
\label{sec:Pse_Introduction}

A piecewise linear- and a smooth structure on a manifold $M$ are called \emph{compatible} with each other, if there exists a triangulation of $M$ as a piecewise linear manifold all of whose simplices are smoothly embedded with respect to the smooth structure. Due to a theorem by Whitehead every smooth manifold $M$ admits a unique compatible piecewise linear structure \cite{Whitehead,MR0198479}. An equivariant version of this result for smooth actions of finite groups on $M$ holds by a theorem of Illman \cite{MR0500993}. Conversely, any piecewise linear manifold of dimension $n\leq 7$ admits a \emph{smoothing}, i.e. a compatible smooth structure \cite{MR0415630,MR0645390,MR0488059,MR0229250,MR0148075}. We show that piecewise linear manifolds in dimension $n \leq 4$ can be equivariantly smoothed in the following sense.
\begin{thmw} Let $M$ be a piecewise linear manifold of dimension $n\leq 4$ on which a finite group $G$ acts by piecewise linear homeomorphisms. Then there exists a compatible $G$-equivariant smooth structure on $M$, i.e. a smoothing with respect to which $G$ acts smoothly.
\end{thmw}
In dimension three our proof solves a challenge by Thurston \cite[p.~208]{MR1435975}. In dimension four the result confirms a conjecture by Kwasik and Lee \cite{MR948910} and guarantees in addition the compatibility condition. In dimension higher than four the statement of the theorem is false, even without the compatibility condition (cf. \cite{MR948910} or Section \ref{sub:higher_dim} below). 

It is always possible to find a piecewise linear triangulation of $M$ with respect to which the group $G$ acts simplicially (cf. Section \ref{sub:pl_structure}). For $n=1$ we can choose the lengths of the segments of such a triangulation so that $G$ acts isometrically. In this way we obtain a desired smoothing. In a similar manner we obtain a canonical smoothing in the case $n=2$ away from the vertices of a triangulation. After a modification in neighbourhoods of the vertices it is possible to extend the smoothing to the whole complex (cf. Section \ref{sub:2-manifold}). In the case $n=3$ Thurston remarks that one probably needs some ``heavy machinery such as the uniformisation theorem for Riemannian metrics on $S^2$, used with ingenuity''. The uniformisation theorem implies that smooth actions of finite groups on $S^2$ are smoothly conjugate to linear actions (cf. Section \ref{sub:linearization}). The corresponding property for $S^3$ is assumed in \cite{MR948910} when conjecturing our result in dimension $4$ without the compatibility condition and has later been proven by Dinkelbach and Leeb \cite{MR2491658} using Ricci flow techniques. Indeed, it turns out that the key ingredients for a proof of our result in dimension three and four are uniqueness of smoothings and the linearisability of finite smooth group actions on spheres up to dimension three. Given these ingredients it seems that our result is rather to be expected, but we believe that having it explicitly written out in the literature can be useful. In fact, we apply it in the proof of the only if direction of the following characterisation. A finite subgroup $G<O(n)$ is generated by transformations $g\in G$ with $\mathrm{rank}(r-I)=2$ if and only if $\R^n/G$ is a piecewise linear manifold \cite{LaMik, La, Lange}.

\section{Preliminaries}

\subsection{Piecewise linear spaces}
\label{sub:pl_structure}

In this section we prove a statement (Proposition \ref{prp:triangulation}) that enables us to give a more workable formulation of our main result (cf. Section \ref{sec:main_result}). First we remind of some concepts from piecewise linear topology. For more details we refer to \cite{MR0248844,MR0350744}. A subset $P\subset \R^n$ is called a \emph{polyhedron} if, for every point $x \in P$, there exists a finite number a simplices contained in $P$ such that their union is a neighbourhood of $x$ in $P$. An open subset of a polyhedron is again a polyhedron. Every polyhedron $P$ in $\R^n$ is the underlying space of some (locally finite) simplicial complex $K$ \cite[Lem.~3.5]{MR0248844}. Such a complex is called a \emph{triangulation} of $P$. A continuous map $f:P\To Q$ between polyhedra $P\subset \R^n$, $Q\subset \R^m$ is called \emph{piecewise linear} (\emph{PL}), if its graph $\{(x,f(x))|x \in P\} \subset \R^{m+n}$ is a polyhedron. It is called \emph{PL homeomorphism}, if it has in addition a PL inverse. This is the case if and only if there exist triangulations of $P$ and $Q$ with respect to which $f$ is a simplicial isomorphism \cite[p.~84, Thm.~3.6.C]{MR0248844}. A polyhedron $P$ is called a \emph{PL manifold} (\emph{with boundary}) of dimension $n$, if every point $p\in P$ has an open neighbourhood in $P$ that is PL homeomorphic to $\R^n$(or to $\R_{\geq 0} \times \R^{n-1}$). If a simplicial complex $K$ triangulates a polyhedron $P$, then $P$ is a PL $n$-manifold, if and only if the link of every vertex of $K$ is a PL $(n-1)$-sphere, i.e. PL homeomorphic to $\partial \Delta^n$. PL manifolds can also be defined as abstract spaces with a \emph{PL structure} \cite[Ch. 3]{MR0248844}. However, every such space can be realised as a polyhedron in some $\R^N$ \cite[Lem.~3.5, p.~80]{MR0248844}\cite[Thm.~7.1, p.~53]{MR0488059}. The following statement is certainly known, but the author has not found a reference.

\begin{prp}\label{prp:triangulation}
A piecewise linear manifold $M$ on which a finite group $G$ acts by piecewise linear homeomorphisms can be triangulated by a simplicial complex $K$ such that $G$ acts simplicially on $|K|$, i.e. it maps simplices linearly onto simplices.
\end{prp}

To prove it we need the notion of a (locally finite) \emph{cell-complex} (cf. \cite{MR0350744}). It can be defined like a simplicial complex, but it is not built up merely from simplices, but more generally from compact convex polyhedra, the so-called \emph{cells} (cf. \cite[pp.~14-15]{MR0350744}). The linear image of a cell is again a cell and the intersection $K\cap L = \{A\cap B|A\in K, B\in L \}$ of two cell complexes $K$ and $L$ is again a cell complex. A \emph{subdivision} of a cell complex $K$ is a simplicial complex $\tilde{K}$ such that $|K|=|\tilde{K}|$ and such that every simplex of $\tilde{K}$ is contained in a cell of $K$. The \emph{$k$-skeleton} $K_{(k)}$ of a cell complex $K$ is the cell complex comprising all cells of $K$ of dimension smaller or equal to $k$. We have
\begin{lem}\label{lem:subdivision_cell_simplex}
Any cell complex can be canonically subdivided into a simplicial complex.
\end{lem}
\begin{proof} The $1$-skeleton of $K$ is already a simplicial complex. We successively subdivide the $2,\ldots,n$-skeleton of $K$. Assume that we have already subdivided the $2,\ldots,k$-skeletons of $K$. Then we \emph{star} the $(k+1)$-cells of its $(k+1)$-skeleton at their barycenters, i.e. we replace each $(k+1)$-cell $C$ by the simplicial complex obtained as the join of the boundary $\partial C$, which is already a simplicial complex, with the barycenter of $C$ (cf. \cite[p.~15]{MR0350744}). Having replaced the $n$-cells we arrive at a simplicial complex that subdivides our initial cell complex.
\end{proof}
With the same method one can prove
\begin{lem}\label{lem:extend_subdivision} Let $K_1$ be a subcomplex of a cell complex $K$. Then any simplicial subdivision of $K_1$ can be extended to a simplicial subdivision of $K$.
\end{lem}
Now we can give a proof for Proposition \ref{prp:triangulation}.
\begin{proof}[Proof of Proposition \ref{prp:triangulation}.] We first triangulate $M$ by a simplicial complex $K \subset \R^N$. For each $g \in G$ we can choose a subdivision $L_g$ of $K$ such that $g$ maps simplices of $L_g$ linearly into simplices of $K$ (cf. \cite[Thm.~3.6, B, p.~84]{MR0248844}). Let $L$ be the cell complex obtained by intersecting the cell complexes $L_g$. Then the restriction of each element $g\in G$ to each cell of $L$ is linear. Therefore, the translates $gL$ are again cell complexes. Hence, their intersection $\bigcap_{g\in G} gL$ is a cell complex on which $G$ acts cellularly, i.e. it maps cells linearly onto cells. Now we apply Lemma \ref{lem:subdivision_cell_simplex} to this complex. By construction, the group $G$ acts simplicially on the resulting simplicial complex.
\end{proof}

Finally we fix some notations. Let $K$ be a simplicial complex. We denote its first \emph{barycentric subdivision} (cf. \cite[p.~119]{MR1867354}) by $K^{(1)}$. The \emph{support} $\mathrm{supp}_K(x)$ of a point $x$ in $K$ is defined to be the smallest dimensional simplex of $K$ that contains $x$. For a simplex $\sigma <K$ the \emph{star} $\mathrm{star}_{K}(\sigma)$ is the smallest simplicial complex that contains all simplices of $K$ that contain $\sigma$. The link of $\sigma$ is defined to be
\[
	\mathrm{lk}_{K}(\sigma) =\{\sigma' \in K | \sigma \cap \sigma' = \emptyset, \exists \tau \in K : \sigma, \sigma'<\tau \}. 
\]
We write $\mathrm{star}_{K}(\sigma)$ and $\mathrm{lk}_{K}(\sigma)$ interchangeably for the star and the link as a simplicial complex and as their underlying space. Also we sometimes omit the index $K$ if its meaning is clear. For a topological space $X$ we denote by $\overline{C}X$ its \emph{closed cone} defined as $(X\times [0,1])/(X\times \{0\})$. For a compact simplicial complex $K$ (in $\R^n$) its closed cone $\overline{C}K$ is again naturally a simplicial complex  (in $\R^{n+1}$).

\subsection{Piecewise differentiable maps and smoothings}\label{sec:smoothing_theory}
The following definition is central for comparing piecewise linear and smooth spaces.
\begin{dfn}
We call a map $f: P \To M$ from a polyhedron $P$ to a smooth manifold with or without boundary $M$ \emph{piecewise differentiable} or \emph{PD}, if there exists a triangulation $K$ of $P$ such that the restriction of $f$ to each simplex is smooth. We call $f$ a \emph{PD homeomorphism (embedding)}, if it is moreover a homeomorphism (onto its image) and each simplex is smoothly embedded, i.e. for each simplex $\sigma \in K$ and each point $p \in \sigma$ the differential $(df_{|\sigma})_p$ is injective. 
\end{dfn}

A smooth structure on a PL manifold with boundary $M$ is called \emph{compatible} with the PL structure of $M$, if the identity map from $M$ as a PL manifold to $M$ as a smooth manifold is a PD homeomorphism. A compatible smooth structure on $M$ is called a \emph{smoothing}. For the proof of our result we need the fact that smoothings in dimensions $n\leq3$ are unique up to diffeomorphism \cite[Thm. 3.10.9, p.~202]{MR1435975} (in fact, we only need this statement for $S^n$, $n\leq3$, cf. Lemma \ref{lem:rad_linearization}). In the case $n=3$ such a proof relies on the fact that every diffeomorphism of $S^2$ can be extended to a diffeomorphism of the corresponding unit ball (cf. \cite[Thm. 3.10.11, p.~202]{MR1435975}).

\subsection{Approximating PD maps by PL maps}\label{sec:PL_approx}

In this section we explain how PD maps can be approximated by PL maps. This will be needed in the proof of our main result (cf. Lemma \ref{lem:rad_extension}).
\begin{dfn}Two PD maps $f,\tilde{f}: P \To \R^n$ are called \emph{$C^1$ $\delta$-close}, if there exists a triangulation $K$ of $P$ such that for every simplex $\sigma \in K$ both $f_{|\sigma}$ and $\tilde{f}_{|\sigma}$ are smooth and the values of $(f-\tilde{f})_{|\sigma}$ and their first derivatives are bounded by $\delta$.
\end{dfn}

The following statement follows immediately from \cite[Thm. 8.8, p.~84, Thm. 8.4, p.~81]{MR0198479}.

\begin{thm}\label{thm:PD_approx_close} Let $f: P \To M \subset \R^n$ be a PD homeomorphism from a compact polyhedron to a smooth connected submanifold $M$ of $\R^n$. Then there exist some $\delta>0$ such that every PD map $\tilde{f}:P \To M$ that is $C^1$ $\delta$-close to $f$ is also a PD homeomorphism.
\end{thm}

In order to approximate PD maps by PL maps we need the following concept (cf. \cite[p.~90]{MR0198479}).

\begin{dfn}
Let $\tilde{K}$ be a subdivision of  $K$ and let $f:K \To \R^n$ be a PD map. The \emph{secant map} $L_{\tilde{K}} f:K \To \R^n$ is defined to be the map that is linear on the simplices of $\tilde{K}$ and coincides with $f$ on the vertices of $\tilde{K}$.
\end{dfn}

By definition $L_{\tilde{K}} f$ is a PL map. For a finite simplicial complex $K$ on which a finite group $G$ acts simplicially, we would like to find $G$-subdivisions $\tilde{K}$ (i.e. subdivisions on which $G$ acts simplicially) such that $L_{\tilde{K}} f$ becomes close to $f$ in the $C^1$ sense. According to \cite[Lem.~9.3, p.~90]{MR0198479}, it is sufficient to find $G$-subdivisions $\tilde{K}$ of $K$ whose simplices' diameters tends to zero while their thickness stays bounded from below. The thickness of a simplex is defined to be the ratio of the minimal distance of its barycenter to its boundary and its diameter. The proof in \cite[Lem.~9.4, p.~92]{MR0198479} of the non-equivariant version of this statement also works in the equivariant case (cf. \cite[Lem.~70]{Lange}), i.e. we have 

\begin{lem}\label{lem:fine_subdivisions} Let $K$ be a finite simplicial complex on which a finite group $G$ acts simplicially. There is a $t_0>0$ such that $K$ has arbitrarily fine $G$-subdivisions for which the minimal simplex thickness is at least $t_0$.
\end{lem}

As in \cite[Thm.~9.6, p.~94]{MR0198479} we immediately obtain
\begin{thm}\label{thm:PD_approx}
Let $K$ be a finite simplicial complex on which a finite group $G$ acts simplicially and let $f:K \To \R^n$ be a PD map. Then for every $\delta > 0$ there exists a $G$-subdivision $\tilde{K}$ of $K$ such that the secant map $L_{\tilde{K}} f$ is $C^1$ $\delta$-close to $f$.
\end{thm}

\subsection{Linearizing smooth actions of finite groups on spheres}\label{sub:linearization} Using \emph{Ricci flow} techniques Dinkelbach and Leeb showed that any smooth action of a finite group $G$ on $S^3$ is smoothly conjugate to an orthogonal action \cite{MR2491658}. The same statement is true for smooth actions of finite groups on $S^2$, but in this case it follows more elementary by the \emph{geometrization of spherical 2-orbifolds} \cite{MR2883685,MR2954692} or by the \emph{uniformization theorem}: Average an arbitrary Riemannian metric on $S^2$ to obtain a $G$-invariant Riemannian metric $g$ on $S^2$. The metric $g$ determines a complex structure on $S^2$ (cf. \cite{MR0074856}) with respect to which $G$ acts biholomorphically. By the uniformisation theorem there exists a biholomorphism to the Riemann sphere (cf. e.g. \cite[Thm.~27.9]{MR0648106}) and thus a smooth function $\phi$ on $S^2$ such that $g_1=e^{\phi}g$ has constant sectional curvature $1$. This function satisfies the equation
\[
 2\Delta_g \phi + \mathrm{S}(g)=\mathrm{S}(g_1) e^{2\phi}
\]
where $\mathrm{S}(g)$ is the curvature of $g$, $\mathrm{S}(g_1)=1$ and $\Delta_g$ denotes the Laplace operator attached to $g$ (cf. \cite[II.3, p.~726]{MR882712}). Hence $\phi$ is unique by the maximum principle. Because of $\mathrm{S}(e^{(\phi\circ h)}g)=\mathrm{S}(h^*g_1)=\mathrm{S}(g_1)\circ h$ for each $h\in G$ the metric $e^{(\phi\circ h)}g$ has constant sectional curvature $1$ on $S^2$ as well for each $h\in G$. This implies that $\phi$ is $G$-invariant by the uniqueness statement above. Hence, $G$ acts isometrically with respect to $g_1$ and its action on $S^2$ can thus be smoothly conjugated to an orthogonal action on the standard unit sphere $S^2 \subset \R^3$.

\subsection{Gluing smoothed PL manifolds} \label{sub:gluing}
Let $P_1$ and $P_2$ be two PL manifolds with boundary endowed with smoothings. Suppose there exists a piecewise linear diffeomorphism $f:\partial P_1 \To \partial P_2$. Then there exists a smooth structure on $P_1 \cup_f P_2$ with respect to which $P_1$ and $P_2$ are smoothly embedded (cf. \cite[Thm.~4.1, p.~25; Remark, p.~24]{MR0190942}). Moreover, using the methods described in Section \ref{sub:pl_structure} one can find triangulations of $P_1$ and $P_2$ whose simplices are smoothly embedded and with respect to which the map $f$ is a simplicial isomorphism. Such triangulations give rise to a triangulation of $P_1 \cup_f P_2$ whose simplices are smoothly embedded. Hence, the smooth structure on $P_1 \cup_f P_2$ above in fact defines a smoothing.

\section{Proof of the main result}\label{sec:main_result}
Let $K$ be a simplicial complex that is also a PL manifold of dimension $n\leq 4$ and let $G$ be the group of all simplicial isomorphisms of $K$. We are going to show that the complex $K$ admits a $G$-equivariant smooth structure and a subdivision all of whose simplices are smoothly embedded. In view of Proposition \ref{prp:triangulation} on the existence of equivariant triangulations, this will imply our main result. Perhaps after taking the first barycentric subdivision of $K$ we can assume that a simplex $\sigma$ of $K$ invariant under some $g \in G$ is pointwise fixed by $g$.

We endow $K$ with an auxiliary polyhedral metric such that all edges have unit length and such that all simplices of $K$ are flat. For $n=1$ we can isometrically identify the resulting metric space with a distance circle in $\R^2$ or a real line to obtain a desired equivariant smoothing. For $n\geq2$ this strategy does not work. We can put a smooth structure on the complement $K^*$ of the $(n-2)$-skeleton of $K$ in $K$ such that every isometry between a subset of $K^*$ and a subset of $\R^n$ is smooth. However, in general it is not possible to extend it to a \emph{compatible} smooth structure on $K$.

In order to extend the smoothing we have to change the smooth structure on $K^*$ in a small neighbourhood of the $(n-2)$-skeleton. Let us begin with the simplest case.

\subsection{Proof for $2$-manifolds} \label{sub:2-manifold}

The canonical metric on $K$ introduced above is induced by a canonical piecewise flat Riemannian metric. Suppose a vertex $x$ of $K$ is contained in $n$ $2$-simplices of $K$. Then we can embed $\mathrm{star}_K(x)$ as a regular $n$-gon of radius $1$ into $\R^2$. Using the cone parameter of $\mathrm{star}_K(x)=\overline{C} \mathrm{lk}_K(x)$ and a smooth cut-off function, in a neighbourhood of $x$ in $K$  we can interpolate between the Riemannian metric induced from the embedding above and the canonical piecewise Riemannian metric on $K$. Doing this for all vertices, we obtain a new equivariant piecewise Riemannian metric on $K$ that coincides with the canonical piecewise Riemannian metric away from the vertices. Close to the vertices and away from the edges the metric defines an equivariant smoothing. Along the interior of the edges we can use the exponential map to define collars. These collars in turn define charts that extend the equivariant smoothing to all of $K$. In view of our proof in higher dimensions note that the same method works, if we start with a piecewise flat metric on $K$ distinct from the canonical one.

What we did is to first construct an equivariant \emph{welding} of $K$ via the metric, i.e. an equivariant and continuous choice of linearisations of the tangent spaces of $K$ (cf. \cite[Def.~3.10.4]{MR1435975}), and then an equivariant smoothing from it, thereby adopting the proof from the non-equivariant case (cf. \cite[Prop.~3.10.7]{MR1435975}). In the non-equivariant case this approach also works in dimension three \cite[Prop.~3.10.7; Thm.~3.10.8]{MR1435975}. In the equivariant case one could use the result in \cite{MR0296808}, which yields a realisation of a simplicial $2$-sphere in $\R^3$ such that all simplicial isomorphisms are realised by isometries, to first construct an equivariant welding and then an equivariant smoothing from it. However, this approach does not easily generalise to dimension four, where additional issues arise when trying to extend the welding from the vertices over the $1$-skeleton (cf. \cite[Challenge.~3.10.17]{MR1435975}) and where there is no analogue of \cite{MR0296808} (cf. \cite{MR0296808}). Instead, we follow a suggestion by Thurston (cf. Introduction and \cite[Challenge~3.10.20]{MR1435975}) in a way that generalises to the four dimensional case. In short, we start with the canonical smooth structure on the $(n-2)$-skeleton of $K$ and successively extend it to smoothings on the complements of neighbourhoods of the $i$-skeleton, $i=n-3,\ldots,0$, and finally to all of $K$. The main ingredients for our proof are provided in the next section.

\subsection{Radially extending equivariant smoothings}\label{sub:rad_ext_smo} 

We will need the following two lemmas to extend equivariant smoothings. The first lemma just formulates the uniqueness of smoothings and the linearisability of finite group actions on spheres in dimensions $n\leq 3$ in a suitable manner (cf. Section \ref{sub:linearization} and Section \ref{sec:smoothing_theory}).

\begin{lem}\label{lem:rad_linearization} Let $K$ be a triangulated PL $n$-sphere, $n\leq3$, on which a finite group $G$ acts simplicially. Suppose that $K$ is equipped with a smoothing with respect to which $G$ acts smoothly. Then there exists a group homomorphism $r : G \To O(n+1)$ and a $G$-equivariant PD homeomorphism $f:K \To S^n\subset \R^{n+1}$.
\end{lem}

The second lemma enables us to extend equivariant smoothings.

\begin{lem}\label{lem:rad_extension} Let $K$ be a triangulated PL $n$-sphere on which a finite group $G$ acts simplicially. Suppose there exists a group homomorphism $r : G \To O(n+1)$ and a $G$-equivariant PD homeomorphism $f:K \To S^n\subset \R^{n+1}$. Then this PD homeomorphism can be extended to a $G$-equivariant PD homeomorphism $\overline{F}: \overline{C} K \To B^{n+1}$ from the closed cone of $K$ to the unit $(n+1)$-ball in $\R^{n+1}$.
\end{lem}
\begin{proof}
\begin{figure}
	\centering
		\includegraphics[width=0.2\textwidth]{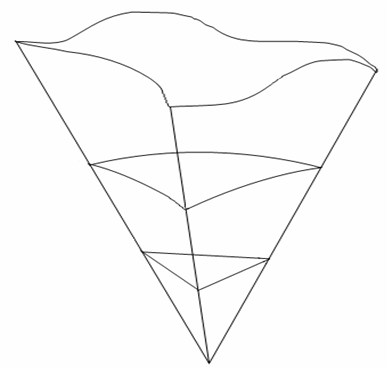}
	\caption{Image of a simplex of the triangulation of $K$ under the map $\overline{F}:CK \To \R^{n+1}$ (cf. Lemma \ref{lem:rad_extension})}
	\label{fig:isotopy}
\end{figure}
We cannot simply extend the map $f$ linearly to the origin, because then the restriction to a simplex could be degenerate at the cone point. However, we can isotopy $f$ along the radial direction to a map that can be linearly extended to the cone point in a compatible way. We do this in two steps. In the first step we isotopy $f$ such that the embedded simplices of $K$ become spherical simplices. The second isotopy deforms the spherical simplices into flat simplices (cf. Figure \ref{fig:isotopy}). More precisely, let $p: \R^{n+1}\backslash\{0\} \To S^n$ be the radial projection and let $\theta: \R \To \R$ be a smooth cut-off function with $0\leq\theta(t)\leq1$, $\theta'(t)\leq 0$, $\theta(t)=1$ for $t<1/3$ and $\theta(t)=0$ for $2/3<t$. According to Theorem \ref{thm:PD_approx} there is a $G$-equivariant subdivision $\tilde{K}$ of the complex $K$ such that the map $L_{\tilde{K}}f:K \To \R^{n+1}$ is close to $f$ in the $C^1$-sense and hence, the same is true for the map $\tilde{f}=p\circ L_{\tilde{K}}f:K \To S^n$. For a sufficiently good approximation the map
\[
	\begin{array}{cccl}
		  F_t : & K 				& \To  	& S^{n} \\
		        & x	& \MTo  & \frac{\theta_1(t)\tilde{f}(x)+(1-\theta_1(t))f(x)}{\left\| \theta_1(t)\tilde{f}(x)+(1-\theta_1(t))f(x) \right\|}
	\end{array}
\]
with $\theta_1(t)=\theta(2t-1)$ is well-defined for $t>1/2$. Moreover, if we choose a sequence of subdivisions such that $\tilde{f}$ converges to $f$ in the $C^1$-sense, then $F_t$ converges uniformly in $t$ to $f$ in the $C^1$-sense. Therefore, for a sufficiently fine subdivision the map
\[
	\begin{array}{cccl}
		  F : & K \times (1/2,1]				& \To  	& \R^{n+1} \\
		         & (x,t) 	& \MTo  & t \cdot F_t(x)
	\end{array}
\]
defines a $G$-equivariant PD embedding by Theorem \ref{thm:PD_approx_close} applied to the maps $F_t$. With $\theta_2(t)=\theta(2t)$ we set
\[
	\begin{array}{cccl}
		  \mu : & K \times [0,1/2]				& \To  	& \R \\
		         & (x,t) 	& \MTo  & \theta_2(t) \frac{1}{\left\| \tilde{f}(x) \right\|} + (1-\theta_2(t) )
	\end{array}
\]
and define
\[
	\begin{array}{cccl}
		  F : & K \times [0,1/2]				& \To  	& \R^{n+1} \\
		         & (x,t) 	& \MTo  & t\mu(x,t)\tilde{f}(x).
	\end{array}
\]
Then the map $F: K \times [0,1] \To B^{n+1}$ descends to a $G$-equivariant PD homeomorphism $\overline{F}: \overline{C}K \To \R^{n+1}$.
\end{proof}

\subsection{Product neighbourhoods} \label{sub:technical}
Before continuing the actual proof, we introduce some organising notations. We denote the set of vertices of $K^{(1)}$, that is of the first barycentric subdivision of $K$, whose supporting simplex in $K$ has dimension $i$ by $v'_i(K)$. In particular, we denote the set of vertices of $K$ by $v(K)=v_0'(K)$. We set $v'(K)=\bigcup_{i=0,\ldots,n} v'_i(K)$ where $n=\mathrm{dim}(K)$. Each $x \in v'(K)$ has an open neighbourhood $U_x\subset \mathrm{star}_{K^{(1)}}(x)$ that splits isometrically as a product $V_x \times S_x$ of connected open sets $V_x \subset \mathrm{supp}_K(x)$ and $S_x \subset \mathrm{supp}_K(x)^{\bot_x}$. Here $\mathrm{supp}_K(x)^{\bot_x}$ is the set of points $y \in \mathrm{star}_K(\mathrm{supp}_K(x))$ for which the straight line between $x$ and $y$ meets $\mathrm{supp}_K(x)$ orthogonally. Note that open sets $U_{x}\subset \mathrm{star}_{K^{(1)}}(x)$ and $U_{y} \subset \mathrm{star}_{K^{(1)}}(y)$ are disjoint for distinct $x,y \in v'_i(K)$.

\begin{dfn}\label{dfn:unifrom-cover} A neighbourhood $U_x\subset \mathrm{star}_{K^{(1)}}(x)$ of $x\in v'(K)$ as above is called a \emph{product neighbourhood}. We call it a \emph{symmetric product neighbourhood} if $U_x$ is in addition invariant under all simplicial isomorphisms of $\mathrm{star}_K(\mathrm{supp}_K(x))$ that leave $\mathrm{supp}_K(x)$ invariant. An open cover $\mathcal{U}=\{U_{x}\}_{x\in v'(K)}$ of $K$ consisting of symmetric product neighbourhoods $U_{x}$ of $x\in v'(K)$ is called a \emph{symmetric product cover}, if for all $i=0,\ldots,n$ and all $x,y \in v'_i(K)$ any simplicial isomorphism between $\mathrm{star}_K(\mathrm{supp}_K(x))$ and $\mathrm{star}_K(\mathrm{supp}_K(x))$ that maps $\mathrm{supp}_K(x)$ onto $\mathrm{supp}_K(x)$, maps $U_x$ onto $U_y$.
\end{dfn}

Note that a symmetric product cover of $K$ is in particular invariant under all simplicial isomorphisms of $K$. In order to have control on the sizes of product neighbourhoods of a symmetric product cover $\mathcal{U}$ we introduce its fineness $\mathrm{fin}(\mathcal{U})$ defined as
\[
		\mathrm{fin}(\mathcal{U}):= \max_{U_x=V_x\times S_x \in \mathcal{U}} \mathrm{inf}\{r>0|S_x\subset B_r(x)\}
\]
and its cofineness $\mathrm{cofin}(\mathcal{U})$ defined as
\[
		\mathrm{cofin}(\mathcal{U}):= \max_{U_x=V_x\times S_x \in \mathcal{U}} \mathrm{inf}\{r>0|V_x\subset B_r(x)\}.
\]
A symmetric product cover with small fineness has large cofiness and vice versa. Clearly, symmetric product covers with arbitrarily small (co)fineness exist.

\subsection{Proof for $3$-manifolds} \label{sub:3manifolds} Let $\mathcal{U}$ be a symmetric product cover of $K$ with small fineness. For a point $x \in v'_1(K)$ on an edge of $K$ we set $S_x^*=S_x \backslash \{x\}$ where $U_x=V_x \times S_x \in \mathcal{U}$ as in the preceding section. The set $V_x\times S^*_x$ inherits a smoothing from $K^*$ that respects the product structure and is invariant under all isometries in $G$ that fix $\mathrm{supp}_{K}(x)$ pointwise. As in the $2$-manifold case we obtain a smoothing of $S_x$ invariant under these isometries that differs from the smoothing of $S^*_x$ only in a small neighbourhood of $x$. Working with representatives of $G$-orbits in $v'_1(K)$ we obtain a $G$-equivariant smoothing of $K_1^{*}=K \backslash \overline{N_{\varepsilon_1}(K_{(0)})}$, the complement of small closed balls around the vertices of $K$. With $\mathrm{fin}(\mathcal{U})$ tending to zero, we can choose $\varepsilon_1$ arbitrarily small. By construction, intersections of simplices of $K$ with $K_1^{*}$ are smoothly embedded.
\begin{figure}
	\centering
		\includegraphics[width=0.5\textwidth]{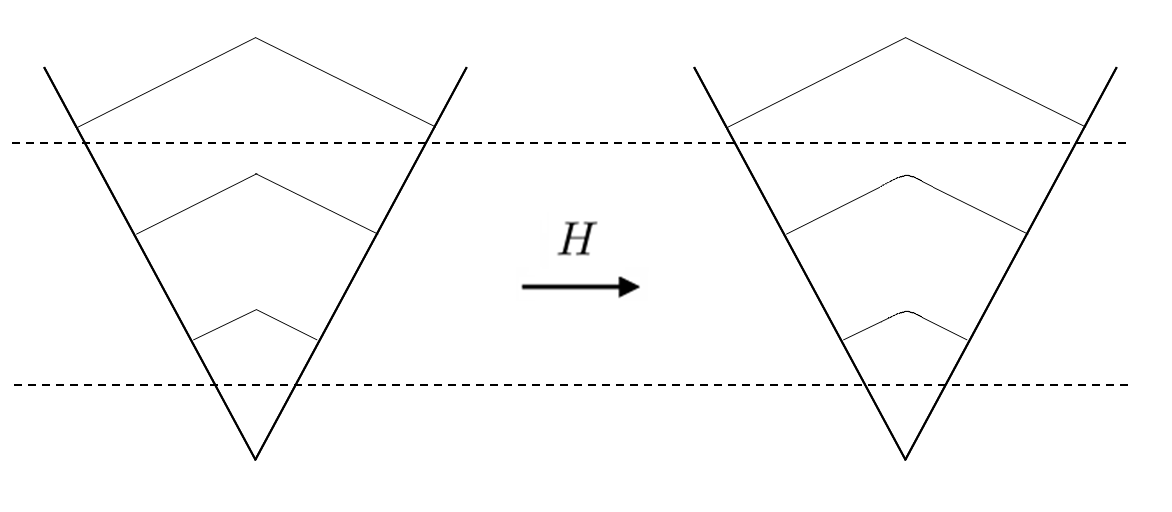}
	\caption{Two-dimensional sketch of the map $H$ restricted to a simplex of $K$. Above the lower dotted line the smooth structure of $K_1^{*}$ is defined (cf. proof). Above the upper dotted line the map $H$ is the identity.}
	\label{fig:new_h}
\end{figure}

For a vertex $x \in v(K)$ of $K$ let $B_x= \{y \in K | d(x,y)=\mathrm{min}_{z \in v(K)} d(y,z)\}$ be a Voronoi domain about $x$. It is a polyhedral $3$-ball in $\mathrm{star}_{K}(x)$ invariant under all simplicial isomorphisms of $\mathrm{star}_{K}(x)$. Its bundary $P_x=\partial B_x$ projects homeomorphically to $\mathrm{lk}_{K}(x)$ with respect to the radial projections in $\mathrm{star}_{K}(x)$. In the present situation, in which all edges of $K$ have unit length, we simply have $B_x= \mathrm{star}_{K^{(1)}}(x)$ and $P_x= \mathrm{lk}_{K^{(1)}}(x)$. We identify $\mathrm{star}_{K}(x)$ with a subset of the cone $CP_x$ and work with cone coordinates $t \cdot v:=(t,v) \in \R_{\geq 0} \times P_x$ to describe points in $\mathrm{star}_{K}(x)$.

We want to change the $G$-equivariant smoothing of $K_1^{*}$ in neighbourhoods of the vertices of $K$ such that for some small $\lambda$ and each vertex $x$ the polyhedron $\lambda \cdot P_x$ is a smooth submanifold. This would induce an equivariant smoothing of $\lambda \cdot P_x$ that could be extended to an equivariant smoothing of $\lambda \cdot B_x$ using Lemma \ref{lem:rad_linearization} and Lemma \ref{lem:rad_extension}. We could then glue together the smoothed balls $\lambda \cdot B_x$ and their complement in $K_1^{*}$ as explained in Section \ref{sub:gluing} to obtain a smoothing of $K$. Moreover, by working with representatives of $G$-orbits in $v(K)$, we could guarantee that the obtained smoothing is equivariant.

We claim that if the fineness of $\mathcal{U}$ is sufficiently small, then there is some small $\lambda$ with the following property. For each vertex $x$ of $K$ there exists an equivariant PD embedding
\[
	H: N_{2\varepsilon_1}(x)^{C} \cap \mathrm{star}_{K}(x)  \To K_1^{*}
\]
of the form $H(t,v)=(\varphi(t,v),v)$ that differs from the identity only for small $t$ and away from the $1$-skeleton of $K$ such that $H(\lambda \cdot P_x) \subset K_1^{*}$ is a smooth submanifold. Using such PD embeddings as new charts alters the smoothing of $K_1^{*}$ in a desired way so that our strategy above applies.

Close to an edge of $K$, say $\mathrm{supp}_K(x)$, $x \in v'_1(K)$, where $H$ is supposed to be the identity, the condition that  $H(\lambda \cdot P_x) \subset K_1^{*}$ is a smooth submanifold is automatically fulfilled. Indeed, in these regions the polyhedron $t\cdot P_x \subset K_1^{*} \cap \mathrm{star}_{K}(x)$ factors through an $S_x$-slice with respect to the isometric splitting $U_x=V_x \times S_x \in \mathcal{U}$ and is thus a smooth submanifold of $K_1^{*}$.
Away from the $1$-skeleton of $K$ the smooth structure on $K_1^{*}$ is still the canonical smooth structure we started with. With respect to this smooth structure the construction of the map $H$ is a matter of elementary calculus that can be performed simplex-wise (cf. Figure \ref{fig:new_h} for a $2$-dimensional sketch of the construction of the map $H$ for $P= \mathrm{lk}_{K^{(1)}}(x)$ and Section \ref{sub:details_construction_h} for more details on the construction).

Note that due to the application of Lemma \ref{lem:rad_extension} and the gluing procedure, in neighbourhoods of vertices of $K$ (open subsets of) the simplices of $K$ are in general not smoothly embedded, only those of a subdivision. However, by choosing the fineness of the symmetric product cover $\mathcal{U}$ we started with sufficiently small, it can be arranged that these neighbourhoods are small.

\subsection{Proof for $4$-manifolds} \label{sub:sect_4}
The proof in the $4$-dimensional case works along the same lines as in the $3$-dimensional case. More care has to be taken only due to the necessity of introducing simplicial subdivisions in dimension three.  Let $\mathcal{U}$ be a symmetric product cover of $K$ with small fineness. As in the first step in the $3$-dimensional case, from the canonical smoothing of $K^*$ we obtain an equivariant smoothing of $K_1^{*}=K \backslash \overline{N_{\varepsilon_1}(K_{(1)})}$, the complement of a closed $\varepsilon_1$-neighbourhood of the $1$-skeleton of $K$. The only difference is that in the present case a two-dimensional factor $V_x$ splits off from the product neighbourhoods of $U_x=V_x \times S_x$, $x \in v'_2(K)$. With $\mathrm{fin}(\mathcal{U})$ tending to zero we can choose $\varepsilon_1$ arbitrarily small.

Now let $U_x=V_x \times S_x$, $x \in v'_1(x)$, be a product neighborhhood corresponding to an edge of $K$. The smoothing of $K_1^{*}$ restricts to a product subset of $U_x=V_x \times S_x$ and respects the product structure. Treating the second factor $S_x$ as in the $3$-dimensional case and working with representatives of $G$-orbits in $v'_1(K)$ we obtain an equivariant smoothing of $K_2^{*}=K \backslash N_{\varepsilon_2}(K_{(0)})$, the complement of small balls $N_{\varepsilon_2}(K_{(0)})$ around the vertices of $K$. With $\mathrm{fin}(\mathcal{U})$ and $\varepsilon_1$ tending to zero, we can choose $\varepsilon_2$ arbitrarily small. Note that in a neighbourhood of the edges of $K$ only (open subsets of) simplices of a subdivsion of $K$ are smoothly embedded in $K_2^{*}$. However, by choosing the fineness of our initial symmetric product cover $\mathcal{U}$ small, we can assume that this neighborhood is closely concentrated around the edges of $K$.

Finally, we claim that if the fineness of $\mathcal{U}$ is sufficiently small, then the smoothing can be extended to all of $K$, i.e. over neighbourhoods of the vertices of $K$, by the same method as in the three-dimensional case. More precisely, we claim that in this case there is some $\lambda$ such that for each vertex $x \in v(K)$ and $P_x= \mathrm{lk}_{K^{(1)}}(x)$ there exists an equivariant PD embedding
\[
	H: N_{2\varepsilon_2}(x)^{C} \cap \mathrm{star}_{K}(x)  \To K_2^{*}
\]
of the form $H(t,v)=(\varphi(t,v),v)$ that differs from the identity only for small $t$ and away from the $1$-skeleton of $K$ such that $H(\lambda \cdot P_x) \subset K_2^{*}$ is a smooth submanifold. For details on the construction of this map we refer to the next section. Given such a map $H$, the proof can be concluded as in the three-dimensional case.

\subsection{Construction of the map $H$} \label{sub:details_construction_h}

In the preceding two sections we have employed PD embeddings $H$ on three occasions. In this section we describe their construction. We treat the case $n=4$. The case $n=3$ works analogously but more easily. One only has to note that in the case $n=3$ the PD embeddings $H$ applied in Section \ref{sub:3manifolds} and Section \ref{sub:sect_4} need to be constructed with respect to different polyhedral metrics. 

Let $\Delta_4=\Delta^4$ be a standard simplex with unit edge length and let $\Delta_3=\Delta^3$ be a face of $\Delta_4$. We regard $\Delta_4$ as the subset $(\Delta_3 \times [0,1])/\sim$ of the cone $C\Delta_3$ with vertex $x$. Moreover, we suppose that $C\Delta_3$ is isometrically embedded in $\R^4$. Let $\mathcal{\tilde U}$ be a symmetric product cover of $\Delta_3$ of small cofineness (cf. Section \ref{sub:technical}). Let $P$ be the simplicial complex $P= \mathrm{lk}_{\Delta_4^{(1)}}(x)$, which is the boundary of $\{y \in C \Delta_3 | d(0,y)=\mathrm{min}_{z \in v(\Delta_4)} d(y,z)\}$ in $C \Delta_3$. We identify $P$ with $\Delta_3$ via radial projection in $C\Delta_3$. In particular, the cover $\mathcal{\tilde U}$ gives rise to a cover of $P$ that we also denote by $\mathcal{\tilde U}$. To describe points in $C\Delta_3$ we work with cone coordinates $(t,v)\in \R_{\geq 0} \times P$ corresponding to $t \cdot v \in CP=C \Delta_3$. Using a partition of unity it is easy to construct a PD map $\varphi_0 : \Delta_3 \To (0,1]$ such that the following properties hold
\begin{enumerate}
\item $\varphi_0$ is equivariant with respect to all simplicial isomorphisms of $\Delta_3$.
\item the restrictions of $\varphi_0$ to the stars $\mathrm{star}_{\Delta_3^{(1)}}(v)$, $v \in v(\Delta_3)$, are smooth.
\item $\varphi_0 \leq 1$, $\varphi_0$ is approximately constant and $\varphi_0(v)=1$ for $v \in v(\Delta_3)$.
\item for $v \in \tilde U=\tilde V \times \tilde S \in  \mathcal{\tilde U}$, the value of $\varphi_0$ only depends on the $\tilde V$-component of $v$.
\item the subset $P'=\{\varphi_0(p) \cdot p| p \in P\} \subset \Delta_4$ is a smooth submanifold of $\Delta_4$. In particular, by transversality the intersections of $P'$ with the faces of $C\Delta_3$ are submanifolds.
\end{enumerate}
Note that if $P'$ is a smooth submanifold of $\Delta_4$, then so is $\lambda P'$ for each $\lambda \in (0,1]$. Let $\varepsilon_2$ be as in the preceding section and let $\lambda$ be such that $5\varepsilon_2 <\lambda<<1$. Given a function $\varphi_0$ as above, a PD embedding
\[
	h: C\Delta_3 \backslash N_{2\varepsilon_2} \To C \Delta_3
\]
can be constructed as $h(t,v)=(\varphi(t,v),v)$ with $\varphi(t,v)=\theta(t) t + (1-\theta(t)) t \varphi_0(v)$ where $\theta : \R_{\geq0} \To \R_{\geq0}$ is a smooth cut-off function with $0\leq\theta(t)\leq1$, $\theta'(t)\geq 0$, $\theta(t)=1$ for $t>1/10$ and $\theta(t)=0$ for $t< 2\lambda$. In the situation of the preceding section, for a vertex $x$ of $K$ copies of $h$ can be put together to define an embedding
\[
	H: N_{2\varepsilon_2}(x)^{C} \cap \mathrm{star}_{K}(x)  \To K_2^{*}.
\]
We claim that this map has the desired properties if the cofineness of $\mathcal{\tilde U}$ and the fineness of the symmetric product cover $\mathcal{U}$ used in the preceding section are sufficiently small. The cover $\mathcal{\tilde U}$ induces a symmetric product cover of $\mathrm{lk}_{K}(x)$ that we also denote by $\mathcal{\tilde U}$. First observe that $H$ is PD: We can assume that $U_y \subset C\tilde{U}_z$ for all $z\in v(\mathrm{lk}_{K}(x))$ and $y \in v'_1(K)$ lying on the edge $(x,z)$. On $C\tilde{U}_z\cap \mathrm{star}_{K}(x)$, $z\in v(\mathrm{lk}_{K}(x))$, the map $H$ is the identity by $(iii)$ and $(iv)$ and thus trivially PD. Outside of these sets the simplices of $K$ are smoothly embedded by the above assumption $U_y \subset C\tilde{U}_z$ (cf. end of the second paragraph in Section \ref{sub:sect_4}) and thus $H$ is PD in these regions too, since $h$ is PD (with respect to the standard smooth structure on $C\Delta_3$) and satisfies $h(C\tilde{U}_z^C)\subset C\tilde{U}_z^C$. It remains to show that $H(\lambda \cdot P_x) \subset K_2^{*}$ is a smooth submanifold where $P_x=\mathrm{lk}_{K^{(1)}}(x)$. To see this note that for sufficiently small $\mathrm{cofin}(\mathcal{\tilde U})$ and $\mathrm{fin}(\mathcal{U})$ we have $ N_{2\varepsilon_2}(x)^{C} \cap U_y \cap C \tilde U_z= \emptyset$ for all $z \in v'_i(\mathrm{lk}_{K}(x))$, $i=1,2,3$, and all $y \in v'_j(K) \cap P_x$ with $j\leq i$ and thus that the following holds due to our construction of the smooth structure on $K_2^{*}$. A point $p \in H(\lambda \cdot P_x) \cap C \tilde U_y$, $y \in v'(\mathrm{lk}_{K}(x))$, has an open neighbourhood $U$ that splits isometrically $U = V \times S \subset C \tilde U_y$ as an open subset $V$ of $C (\mathrm{supp}_{\mathrm{lk}(x)}(y))$ and an orthogonal submanifold $S \subset C (\mathrm{supp}_{\mathrm{lk}(x)}(y))^{\bot_z}$, for some $z\in C (\mathrm{supp}_{\mathrm{lk}(x)}(y))$, such that the smooth structure on $K_2^{*}$ restricts to the product smooth structure on $U = V \times S$ of the Euclidean smooth structure on $V$ and the smooth structure on $S$. With respect to the splitting $U = V \times S$ a neighbourhood of $p$ in $H(\lambda \cdot P_x)$ splits as a product of a smooth submanifold of $V$ and an open subset of $S$ by properties $(iv)$, $(v)$ and the choice of $P$ and $P_x$. In particular, this neighbourhood is a smooth submanifold of $U$ and thus of $K_2^{*}$. It follows that $H(\lambda \cdot P_x)$ is a smooth submanifold of $K_2^{*}$ as claimed.

\subsection{Higher dimensions} \label{sub:higher_dim}
There exist piecewise linear actions of $\Z_2$ on a $5$-dimensional piecewise linear sphere that cannot be equivariantly smoothed. One way to obtain such an example is as follows (cf. \cite[p.~260]{MR948910}). The group $\Z_2$ admits a piecewise linear action on $S^4$ whose fixed point set is a knotted $S^2$, i.e. the fundamental group of its complement is distinct from $\Z$ \cite[p.~347]{MR0515288}. By suspending this action one obtains a piecewise linear action of $\Z_2$ on $S^5$ with fixed point set $S^3$. However, this action cannot be equivariantly smoothed because its fixed point set $S^3\subset S^5$ is not locally flat.
\newline
\newline
\emph{Acknowledgments.} I would like to express my sincerest thanks to my advisor Alexander Lytchak for his encouragement and support and to Stephan Stadler and Stephan Wiesendorf for critically listening to my different attempts of proving this paper's result. I would like to thank Fernando Galaz-Garcia for drawing my attention to \cite{MR948910} and everybody else who answered my questions. Parts of this work were written during a visit at the Pennsylvania State University. It is a pleasure to thank Anton Petrunin and the mathematical department there for their hospitality and support.


\begin{thebibliography}{99}

\bibitem{MR882712}
{J. P. Bourguignon and J. P. Ezin.} Scalar curvature functions in a conformal class of metrics and conformal transformations.
\textit{Trans. Amer. Math. Soc.} \textbf{301} (1987), 723--736.

\bibitem{MR0229250}
{J. Cerf.} Sur les difféomorphismes de la sph\`ere de dimension trois $(\Gamma_4=0)$.
\textit{Lecture Notes in Mathematics, No. 53 Springer-Verlag, Berlin-New York} (1968), xii+133 pp.

\bibitem{MR0074856}
{S. Chern.} An elementary proof of the existence of isothermal parameters on a surface.
\textit{Proc. Amer. Math. Soc.} \textbf{6} (1955), 771--782.

\bibitem{MR2883685}
{M. W. Davis.} Lectures on orbifolds and reflection groups.
\textit{Adv. Lect. Math.} \textbf{16} (2011), 63--93.

\bibitem{MR2491658}
{J. Dinkelbach and B. Leeb.} Equivariant Ricci flow with surgery and applications to finite group actions on geometric 3-manifolds.
\textit{Geom. Topol.} \textbf{13} (2009), 1129--1173.

\bibitem{MR0648106}
{O. Forster.} Lectures on Riemann surfaces (Graduate Texts in Mathematics, Springer-Verlag, New York-Berlin, 1981).


\bibitem{MR1867354}
{A. Hatcher.} Algebraic topology (Cambridge University Press, Cambridge, 2002).

\bibitem{MR0415630}
{M. W. Hirsch and B. Mazur.} Smoothings of piecewise linear manifolds (Annals of Mathematics Studies, No. 80, Princeton University Press, Princeton, N. J.; University of Tokyo Press, Tokyo, 1974).

\bibitem{MR0248844}
{J. F. P. Hudson.} Piecewise linear topology (University of Chicago Lecture Notes, W. A. Benjamin, Inc., New York-Amsterdam, 1969).

\bibitem{MR0500993}
{S. Illman.} Smooth equivariant triangulations of $G$-manifolds for $G$ a finite group.
\textit{Math. Ann.} \textbf{233} (1978), 199--220.

\bibitem{MR0148075}
{M. A. Kervaire and B. J. W. Milnor.} Groups of homotopy spheres. I.
\textit{Ann. of Math. (2)} \textbf{77} (1963), 504--537.

\bibitem{MR0645390}
{R. C. Kirby and L. C. Siebenmann.} Foundational essays on topological manifolds, smoothings, and triangulations (Annals of Mathematics Studies, No. 88, Princeton University Press, Princeton, N.J.; University of Tokyo Press, Tokyo, 1977).

\bibitem{MR948910}
{S. Kwasik and L. B. Lee.} Locally linear actions on $3$-manifolds.
\textit{Math. Proc. Cambridge Philos. Soc.} \textbf{104} (1988), 253--260.

\bibitem{LaMik}
{Ch. Lange and M. A. Mikha\^ilova.} Classification of finite groups generated by reflections and rotations.
\textit{Transform. Groups} \textbf{21} (2016), 1155-1201.

\bibitem{La}
{Ch. Lange.} Characterization of finite groups generated by reflections and rotations.
\textit{J. Topol.} doi:10.1112/jtopol/jtw020.

\bibitem{Lange}
{Ch. Lange.} Some results on orbifold quotients and related objects.
\textit{KUPS - K\"{o}lnerUniversit\"{a}tsPublikationsServer} (2016).

\bibitem{MR0296808}
{P. Mani.} Automorphismen von polyedrischen Graphen.
\textit{Math. Ann.} \textbf{192} (1971), 279--303.

\bibitem{MR0190942}
{J. Milnor.} Lectures on the $h$-cobordism theorem (Princeton University Press, Princeton, N.J.,  1965).

\bibitem{MR0488059}
{E. E. Moise.} Geometric topology in dimensions $2$ and $3$ (Graduate Texts in Mathematics, Vol. 47, Springer-Verlag, New York-Heidelberg, 1977).

\bibitem{MR0198479}
{J. R. Munkres.} Elementary differential topology (Lectures given at Massachusetts Institute of Technology, Fall, Princeton University Press, Princeton, N.J., 1966).

\bibitem{MR0515288}
{D. Rolfsen.} Knots and links (Mathematics Lecture Series, No. 7, Publish or Perish, Inc., Berkeley, Calif., 1976).

\bibitem{MR0350744}
{C. P. Rourke and B. J. Sanderson.} Introduction to piecewise-linear topology (Ergebnisse der Mathematik und ihrer Grenzgebiete, Band 69, 1972).

\bibitem{MR1435975}
{W. P. Thurston.} Three-dimensional geometry and topology. Vol. 1 (Princeton Mathematical Series, 1997).

\bibitem{Whitehead}
{J. Whitehead.} On $C^1$ complexes.
\textit{Annals of Math.} \textbf{41} (1940), 809--832.

\bibitem{MR2954692}
{B. P. Zimmermann.} On finite groups acting on spheres and finite subgroups of orthogonal groups.
\textit{Sib. \`Elektron. Mat. Izv.} \textbf{9} (2012), 1--12.

\end{thebibliography}
\end{document}